\documentclass[a4paper,12pt,parskip=half*,numbers=noendperiod]{scrreprt}

\usepackage[utf8]{inputenc}
\usepackage[T1]{fontenc}
\usepackage[english]{babel}
\usepackage{lmodern}
\usepackage{xparse}

\usepackage{lipsum}
\usepackage[shortlabels]{enumitem}
\usepackage{verbatim}
\usepackage[normalem]{ulem}
\usepackage[dvipsnames]{xcolor}
\usepackage{csquotes}
\usepackage{pdfpages}
\usepackage{easy-todo}
\usepackage{setspace}

\usepackage{graphicx}
\usepackage{caption}
\usepackage{subcaption}

\usepackage{mathtools}
\usepackage{stmaryrd}
\usepackage{amsthm, thmtools}
\usepackage{framed}
\usepackage{nameref}
\usepackage[colorlinks,menucolor=blue,linkcolor=blue, citecolor=blue, urlcolor=blue]{hyperref} 
\usepackage[capitalise]{cleveref}
\crefname{subsection}{Subsection}{Subsections}
\Crefname{subsection}{Subsection}{Subsections}
\usepackage{amssymb}
\usepackage[mathscr]{euscript}
\usepackage{esint}
\usepackage{esvect}
\usepackage{relsize}

\usepackage{tikz}
\usepackage{tikz-qtree}
\usepackage[framemethod=tikz]{mdframed}
\usetikzlibrary{cd}
\usetikzlibrary{matrix}
\usetikzlibrary{backgrounds}

\usepackage{float}
\usepackage{cancel}

\newtheorem{theorem}{Theorem}[section]
\newtheorem{proposition}{Proposition}[section]
\newtheorem{lemma}{Lemma}[section]
\newtheorem{corollary}{Corollary}[section]

\theoremstyle{definition}
\newtheorem{definition}{Definition}[section]

\theoremstyle{remark}
\newtheorem*{remark}{Remark}

\numberwithin{equation}{section}

\title{Symplectic structures preserved by geodesic symmetries}
\author{Pierre Bieliavsky and Maxime Willaert\\ UCLouvain (Belgium)}
\date{}

\begin{document}

\maketitle

\begin{abstract}
Answering a conjecture by S. Kobayashi, in 1986, K. Sekigawa and L. Vanhecke proved that an almost hermitian manifold whose local geodesic symmetries preserve the Kähler 2-form  is a locally symmetric hermitian space.

In the present paper, we relax the hermitean hypothesis by only requiring the manifold to be symplectic. In other words,  we study the symplectic manifolds equipped with a symplectic connection whose geodesic symmetries are (local) symplectomorphisms. We call ``S-type'' these affine symplectic manifolds.

\end{abstract}

\tableofcontents

\section{Introduction}

\noindent The local geodesic symmetries of a linear, say torsionfree, connection on a smooth manifold $M$ entirely determine  the connection. Indeed, if $\nabla$ denote the covariant derivative 
in the tangent bundle and if $s_x:U\to U$ ($U$ open set in $M$) denotes the geodesic symmetry at point $x$ in $M$, then (see \cite{Bert:2011}), for all tangent vector fields $X$ and $Y$ on $M$, one has
$$
\nabla_XY\left.\right|_x\;=\;\frac{1}{2}\,\left[X\,,\,Y\;+\;s_{x\star}Y\right]_x\;.
$$
From the above formula, one easily checks that every (geodesic) symmetry invariant tensor field on $M$ is necessarily parallel. The converse of the last statement of course does not hold in general. For instance,
if the metric tensor of a Riemannian manifold is invariant under the geodesic symmetries of the associated Levi-Civita connection, the Riemannian manifold is a locally symmetric space. Observe that, in this last case, the geometric property of the metric being symmetry invariant is encoded by a differential property : the fact that the Riemann curvature tensor $R$ of $(M,g)$ is parallel.

\vspace{2mm}

\noindent In \cite{SkVh:1986}, Sekigawa and Vanhecke proved that an almost Hermitean manifold whose geodesic symmetries preserve the Kähler form is necessarily locally a Hermitian symmetric space.
In particular, this result positively answered (and widely generalized)  a conjecture by Soshishi Kobayashi \cite{SkVh:1986} stating that the class of  homogeneous compact Kahler manifolds whose geodesic symmetries preserve the symplectic structure coincides with the class of compact Hermitean symmetric spaces.

\vspace{2mm}

\noindent It is tempting to investigate the situation within a purely symplectic framework i.e. without the condition on the manifold to be almost Hermitean.  

\noindent In the present article, we define the notion of \emph{S-type symplectic connection} on a symplectic manifold as a symplectic connection whose (local) geodesic symmetries preserve the symplectic structure.

\noindent The main result of the paper is the fact that, when analytic, a connection is of S-type is  equivalent to a recursive sequence of curvature conditions, which we derive by relying on the Jacobi equation to compute the Taylor expansion of the symplectic form in normal coordinates. Namely, we prove the following:
\begin{theorem}\label{THMPRINC}
Let $(M,\omega,\nabla)$ be a Fedosov manifold i.e. $(M,\omega)$ is a symplectic manifold and $\nabla$ is a torsionfree linear covariant derivative in $T(M)$ such that $\nabla\omega=0$. Let $R$ denote the curvature tensor of $\nabla$. 

\noindent For every tangent vector field $X$ on  $M$, consider the associated field of endomorphisms of  $T(M)$ defined by 
$$
\Pi(Y)\;=\;R(X,Y)X\;.
$$
For every integer $r\geq2$, define recursively the following fields of endomorphisms 
\begin{equation*}
    Q^r=(r-1)\nabla^{r-2}_X\Pi+\sum_{q=2}^{r-2}\binom{r-1}{q+1}(\nabla^{r-2-q}_X\Pi)\circ Q^q
\end{equation*}
with initial condition $Q^2:=\Pi$ (the sum being defined to be $0$ for $r=2,3$).

\noindent Finally, for every odd integer $r\geq3$ set :
\begin{equation}
    P^r:=(r+2)Q^r+\sum_{q=2}^{(r-1)/2}\binom{r+2}{q+1}(Q^q)^\top\circ Q^{r-q}
\end{equation}
where ${}^\top$ stands for symplectic adjoint.

\noindent Then, the following hold.

\begin{enumerate}
\item[(i)] If, for every point $x$ in $M$, the symplectic structure $\omega$ is invariant under $s_x$ i.e. locally around $x$:
$$
s_x^\star\omega\;=\;\omega\;,
$$
then,  for every $r=3,5,7,...$, the endomorphism $P^r_x$ belongs to the linear symplectic Lie algebra $\mathfrak{sp}(T_x(M),\omega_x)$.

\item[(ii)] Conversely, if $\nabla$ is analytic at $x$ and such that, for every $X$, $P^r_x$ belongs to the linear symplectic Lie algebra $\mathfrak{sp}(T_x(M),\omega_x)$ for every $r=3,5,7..$,
then, the symplectic structure $\omega$ is (locally) invariant under $s_x$.
\end{enumerate}
\end{theorem}

\noindent We then use this result combined with some in \cite{BourgeoisCahen:1998} and \cite{CahenGuttRawnsley:2000} to prove 
\begin{corollary}
\begin{enumerate}
\item[(i)]
S-type connections are necessarily \emph{preferred} in the sense of Bourgeois and Cahen \cite{BourgeoisCahen:1998}. 
\item[(ii)] Every analytic preferred connection on a \emph{symplectic surface} is S-type. In particular \cite{BourgeoisCahen:1998}, S-type connections are not homogeneous. 

\item[(iii)] The class of S-type connections contains the class of Ricci-type connections. 
\item[(iv)] However, in higher dimensions, S-type connections are generally not of \emph{Ricci-type} in the sense of \cite{CahenGuttRawnsley:2000}. 
\end{enumerate}
\end{corollary}
\noindent We also show that, in contrast with the almost Hermitean situation, S-type connections are generally far from being symmetric symplectic.

\vspace{3mm}

\noindent The figure below summarizes the inclusions and intersections between the different classes of symplectic connections mentioned here.

\vspace{4mm}

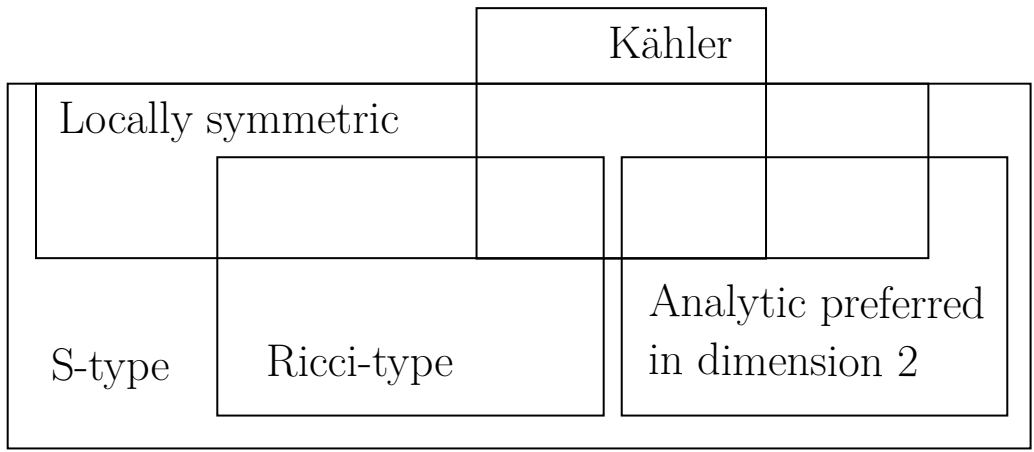
\tikzset{every picture/.style={line width=0.75pt}} 

\begin{tikzpicture}[x=0.75pt,y=0.75pt,yscale=-1,xscale=1]

\draw   (64.64,109) -- (575,109) -- (575,292.5) -- (64.64,292.5) -- cycle ;
\draw   (298.57,71) -- (443,71) -- (443,197) -- (298.57,197) -- cycle ;
\draw   (78.77,109) -- (524,109) -- (524,196.77) -- (78.77,196.77) -- cycle ;
\draw   (371,146) -- (563.38,146) -- (563.38,275.96) -- (371,275.96) -- cycle ;
\draw   (169.22,146) -- (362,146) -- (362,275.96) -- (169.22,275.96) -- cycle ;

\draw (84.79,240.68) node [anchor=north west][inner sep=0.75pt]  [font=\Large] [align=left] {S-type};
\draw (192.09,239.06) node [anchor=north west][inner sep=0.75pt]  [font=\Large] [align=left] {Ricci-type};
\draw (383.15,208.48) node [anchor=north west][inner sep=0.75pt]  [font=\Large,rotate=-0.52] [align=left] {Analytic preferred\\in dimension 2};
\draw (362.73,78.77) node [anchor=north west][inner sep=0.75pt]  [font=\Large] [align=left] {Kähler};
\draw (88.48,116.81) node [anchor=north west][inner sep=0.75pt]  [font=\Large] [align=left] {Locally symmetric};

\end{tikzpicture}

\vspace{2mm}

\noindent We end this introduction by mentioning a possible application in the context of quantization. It is well established that symplectic symmetric spaces are a particularly nice category of spaces in view of quantization.
The main reason for this was intuited by Andr\'e and Juliane Unterberger as well as Alan Weinstein in \cite{Unterberger:1984} and  \cite{WeinsteinTTSSS:1993}. Basically, the geodesic symmetries quantize as unitary involutions inducing in turn a manageable equivariant symbolic calculus. 

\noindent One can reasonably expect that our present context (S-type connections) will allow for similar constructions, therefore extending the symmetric case.

\section{Preliminaries}

\begin{definition}
Let $M$ be a smooth manifold with symplectic form $\omega$. A \textbf{symplectic connection} on $(M,\omega)$ is a torsionless affine connection $\nabla$ on $M$ such that $\nabla\omega=0$. A \textbf{Fedosov manifold} is a symplectic manifold together with a symplectic connection.
\end{definition}

\begin{definition}
    Let $\nabla$ be an affine connection on a smooth manifold $M$. Let $\mathcal{E}\subseteq TM$ denote the domain of the exponential map of $\nabla$. For a point $p\in M$, let $\exp_p:\mathcal{E}_p:=\mathcal{E}\cap T_pM\to M$ denote the restricted exponential map.

    Suppose that $U_p$ is a diffeomorphic image by $\exp_p$ of a star-shaped neighborhood $V_p$ of $0\in T_pM$, $U_p$ is then called a \textbf{normal neighborhood} of $p$.

    We say that $U_p$ is \textbf{symmetric} if $V_p=-V_p$, that is if for any $X\in \mathcal{E}_p$, $\exp_p(X)\in U_p$ implies $\exp_p(-X)\in U_p$. In that case the local \textbf{geodesic symmetry} centered at $p$ is defined to be the unique involutive diffeomorphism $s_p:U_p\to U_p$ such that $(s_p\circ\exp_p)(X)=\exp_p(-X)$ $\forall X\in V_p$.
\end{definition}

\begin{definition}
    A covariant $k$-tensor field $\alpha$ over $M$ is said to be \textbf{symmetric} for an affine connection $\nabla$ if for any point $p\in M$ there exists a symmetric normal neighborhood $U$ of $p$ such that $s^*_p(\alpha|_U)=(-1)^k\alpha|_U$.
\end{definition}

\begin{proposition}
    Let $\alpha$ be a covariant tensor field on $M$ and let $\nabla$ be an affine connection on $M$. If $\alpha$ is symmetric for $\nabla$ and if $\nabla$ is torsionless, then $\nabla\alpha=0$.
\end{proposition}

\begin{definition}
    A torsionless connection $\nabla$ over a symplectic manifold $(M,\omega)$ is said to be of \textbf{type S} (or \textbf{S-type}) if $\omega$ is symmetric for $\nabla$, that is if for any point $p\in M$ there exists a symmetric normal neighborhood $U$ of $p$ such that $s^*_p(\omega|_U)=\omega|_U$. Such a connection is necessarily symplectic.
\end{definition}

\section{Recursive curvature conditions for an S-type connection}\label{sect:RecCdt}

In this section we use the radial power expansion of the symplectic form in normal coordinates to derive a recursive sequence of curvature conditions, which are necessary for a symplectic connection to be of type S, and sufficient when the connection is analytic. The first of these conditions will allow us to prove that an S-type condition must necessarily be preferred.

\subsection{Jacobi fields}

So let $(M,\omega)$ be symplectic manifold with symplectic connection $\nabla$. Fix a point $p$ of $M$ and let $X_p$ be a nonzero element of $T_pM$. We complete $X_p=:b_1$ into a basis $(b_k)_{k=1}^{2n}$ of $T_pM$, and this basis induces a normal coordinate system $(x^1,...,x^{2n})$ over a neighborhood $U$ of $p$. The associated coordinate vector fields over $U$ are denoted by $(\partial_k)_{k=1}^{2n}$, in particular $\partial_k|_p=b_k$ for all $1\leq k\leq 2n$.

For some $\epsilon>0$, we let $\gamma:]-\epsilon,\epsilon[\to U$ be the geodesic with initial velocity $\dot{\gamma}(0)=X_p$, and write $X:=\dot{\gamma}$. For $i\geq 2$, we define the vector field $Z_i(t):=t\partial_i|_{\gamma(t)}$ along $\gamma$. Our first objective is to find an expression for $\nabla_X^rZ_i(0)$ $i\geq 2$, $r\geq 0$ in terms of the covariant derivatives of $R$, the curvature of $\nabla$ (here $\nabla_X$ denotes the covariant derivative along $\gamma$).

Order $0$ and $1$ are easily obtained from the definition $Z_i(t)=t(\partial_i\circ\gamma)(t)$, we have
\begin{equation}\label{eq:Zi0}
    Z_i(0)=0
\end{equation}
and by the Leibniz rule
\begin{equation*}
    \nabla_XZ_i(t)=(\partial_i\circ\gamma)(t)+t\nabla_X(\partial_i\circ\gamma)(t)
\end{equation*}
which implies
\begin{equation}\label{eq:Zi1}
    \nabla_XZ_i(0)=\partial_i|_p.
\end{equation}

To obtain derivatives of higher order, the key observation is that $Z_i$ satisfies the Jacobi equation
\begin{equation*}
    \nabla^2_XZ_i=R(\dot{\gamma},Z_i)\dot{\gamma}
\end{equation*}
indeed, since $Z_i(t)=\left.\frac d{ds}\right|_{s=0}\exp_p(t(X_p+s\partial_i|_p))$, $Z_i$ is an infinitesimal geodesic variation of $\gamma$ (and $\nabla$ is assumed to be torsionless). Introducing the endomorphism field $\Pi:=R(\dot{\gamma},-)\dot{\gamma}$ along $\gamma$, we rewrite the Jacobi equation for $Z_i$
\begin{equation}\label{eq:Jacobi}
    \nabla^2_XZ_i=\Pi Z_i.
\end{equation}
We immediately compute
\begin{equation}\label{eq:Zi2}
    \nabla^2_XZ_i(0)=\Pi|_00=0
\end{equation}
and for $r\geq 0$, applying $\nabla^r_X$ to \ref{eq:Jacobi}, we obtain
\begin{equation}\label{eq:DerJacobi}
    \nabla^{r+2}_XZ_i=\sum_{q=0}^r\binom{r}{q}(\nabla^{r-q}_X\Pi)(\nabla^q_XZ_i).
\end{equation}

\subsection{Recursion operators}
\noindent We start by observing 
\begin{lemma} For $r\geq 2$, one has
\begin{equation}\label{eq:ZiQr}
    \nabla^{r+1}_XZ_i(0)=Q^r_0\partial_i|_p
\end{equation}
where the endomorphism fields $Q^r$ , $r\geq 2$ along $\gamma$ are defined by the recurrence formula
\begin{equation}\label{eq:RecQ}
    Q^r=(r-1)\nabla^{r-2}_X\Pi+\sum_{q=2}^{r-2}\binom{r-1}{q+1}(\nabla^{r-2-q}_X\Pi)\circ Q^q
\end{equation}
(by convention, the sum from $q=2$ to $r-2$ is $0$ for $r=2,3$, allowing us to compute the initial elements $Q^2$ and $Q^3$). 
\end{lemma}
\begin{proof}
    We proceed by induction. Take $r\geq 2$ and assume that for $2\leq q\leq r-2$, $\nabla^{q+1}_XZ_i(0)=Q^q_0\partial_i|_p$ (this is vacuously true for $r=2,3$, allowing us to treat the first two cases), then by applying $\nabla^{r-1}_X$ to the Jacobi equation (i.e. by taking equation \ref{eq:DerJacobi} and replacing $r$ by $r-1$) we obtain
    \begin{equation*}
        \nabla^{r+1}_XZ_i(0)=\sum_{q=0}^{r-1}\binom{r-1}{q}\nabla^{r-1-q}_X\Pi_0\nabla^{q}_XZ_i(0).
    \end{equation*}
    Using $Z_i(0)=\nabla^2_XZ_i(0)=0$, $\nabla_XZ_i(0)=\partial_i|_p$ the induction hypothesis and changing the index of summation from $q$ to $q-1$, this can be rewritten
    \begin{equation*}
        \begin{split}
            \nabla^{r+1}_XZ_i(0)&=\left.\left((r-1)\nabla^{r-2}_X\Pi+\sum_{q=2}^{r-2}\binom{r-1}{q+1}\nabla^{r-2-q}_X\Pi\circ Q^q\right)\right|_0\partial_i|_p\\
            &=Q^r_0\partial_i|_p.
        \end{split}
    \end{equation*}
\end{proof}

\noindent Relying on this formula, we can express the fields $Q^r$, $r\geq 2$ as follows:
\begin{lemma}
\begin{equation}\label{eq:CoeffQ}
    Q^r=\sum c^r_{i_1...i_k}\nabla^{i_1}_X\Pi\circ...\circ\nabla^{i_k}_X\Pi
\end{equation}
where we sum over $k\geq 1$, $i_j\geq 0$ with $r=2k+i_1+...+i_k$ and the constants $c^r_{i_1...i_k}$ are defined by
\begin{equation}\label{eq:Cr1}
    c^r_{i_1}=(r-1)
\end{equation}
for $k=1$ (which implies $i_1=r-2$) and
\begin{equation}\label{eq:CrRec}
    c^r_{i_1...i_k}=\binom{r-1}{i_1}c^{r-2-i_1}_{i_2...i_k}
\end{equation}
for $k>1$. 
\end{lemma}

\begin{proof}
    Here as well, we proceed by induction. Take $r\geq 2$ and assume that equation \ref{eq:CoeffQ} is true for $2\leq q\leq r-2$ (once again this is vacuously true for $r=2,3$). Taking the recurrence formula \ref{eq:RecQ} which defines $Q^r$, changing the index of summation from $q$ (ranging from $2$ to $r-2$) to $i_1:=r-2-q$ (ranging from $0$ to $r-4$), and using $\binom{r-1}{q+1}=\binom{r-1}{r-1-i_1}=\binom{r-1}{i_1}$ we get 
    \begin{equation*}
        \begin{split}
            Q^r&=(r-1)\nabla^{r-2}_X\Pi+\sum_{i_1=0}^{r-4}\binom{r-1}{i_1}\nabla^{i_1}_X\Pi\circ Q^{r-2-i_1}\\
            &=(r-1)\nabla^{r-2}_X\Pi+\sum_{i_1=0}^{r-4}\sum \binom{r-1}{i_1}c^{r-2-i_1}_{i_2...i_k}\nabla^{i_1}_X\Pi\circ \nabla^{i_2}_X\Pi\circ...\circ \nabla^{i_k}_X\Pi
        \end{split}
    \end{equation*}
    where the last sum ranges over $k\geq 2$, $i_2,...,i_k\geq 0$ with $r-2-i_1=2(k-1)+i_2+...+i_k$ or equivalently $r=2k+i_1+...+i_k$. This proves equation \ref{eq:CoeffQ} for our chosen $r$, with the coefficients $c^r_{i_1...i_k}$ defined as in equations \ref{eq:Cr1} and \ref{eq:CrRec}.
\end{proof}
\begin{remark}
\noindent Putting together \ref{eq:Cr1} and \ref{eq:CrRec}, we see that we have
\begin{equation*}
    c^r_{i_1...i_k}=\binom{r-1}{i_1}\binom{r-3-i_1}{i_2}\binom{r-5-i_1-i_2}{i_3}...\binom{r-2k+1-i_1-...-i_{k-1}}{i_k}
\end{equation*}
or equivalently, since $r=2k+i_1+...+i_k$
\begin{equation*}
    c^r_{i_1...i_k}=\binom{2k-1+i_1+...+i_k}{i_1}\binom{2(k-1)-1+i_2+...+i_k}{i_2}...\binom{1+i_k}{i_k}.
\end{equation*}
\end{remark}

\subsection{The main result}

Before proceeding, let us provide a few of the first terms
\begin{equation*}
    \begin{split}
        \nabla^3_XZ_i(0)=Q^2_0\partial_i|_p=&\Pi_0\partial_i|_p\\
        \nabla^4_XZ_i(0)=Q^3_0\partial_i|_p=&2\nabla_X\Pi_0\partial_i|_p\\
        \nabla^5_XZ_i(0)=Q^4_0\partial_i|_p=&\left(3\nabla^2_X\Pi_0+\Pi_0\circ\Pi_0\right)\partial_i|_p\\
        \nabla^6_XZ_i(0)=Q^5_0\partial_i|_p=&\left(4\nabla^3_X\Pi_0+4\nabla_X\Pi_0\circ\Pi_0+2\Pi_0\circ\nabla_X\Pi_0\right)\partial_i|_p\\
        \nabla^7_XZ_i(0)=Q^6_0\partial_i|_p=&\left(5\nabla^4_X\Pi_0+10\nabla^2_X\Pi_0\circ\Pi_0+10\Pi_0\circ\nabla^2_X\Pi_0\right.\\
        &\left.+3\nabla_X\Pi_0\circ\nabla_X\Pi_0+\Pi_0\circ\Pi_0\circ\Pi_0\right)\partial_i|_p.
    \end{split}
\end{equation*}

We will now use the formulae for $\nabla^rZ_i(0)$ $r\geq 0$ we just derived in order to rexpress the coefficients in the Taylor expansion of the components of the symplectic form $\omega$ (in the normal coordinate system $(U,x^1,...,x^{2n})$) in terms of the curvature $R$ (and its covariant derivatives) and $\omega$ itself. More precisely, for $2\leq i<j\leq 2n$ we define the function
\begin{equation}\label{eq:hijDef}
    h_{ij}(t):=\omega_{\gamma(t)}(Z_i(t),Z_j(t))=t^2(\omega_{ij}\circ\gamma)(t)
\end{equation}
for $t\in]\epsilon,\epsilon[$. And for $2\leq i$ we define the function
\begin{equation}\label{eq:h1iDef}
    h_{1i}(t):=\omega_{\gamma(t)}(X(t),Z_i(t))=t(\omega_{1i}\circ\gamma)(t).
\end{equation}
Since $\omega$ is skew-symmetric, this covers all of its components. The first derivatives of the $h_{ij}$ and $h_{1i}$ functions can be computed directly. The presence of a factor $t^2$ in \ref{eq:hijDef} and $t$ in \ref{eq:h1iDef} implies
\begin{equation*}
   h_{1i}(0)=h_{ij}(0)=\left.\frac{d}{dt}h_{ij}(t)\right|_{t=0}=0
\end{equation*}
while from $\nabla\omega=0$, $\nabla_XZ_i(0)=\partial_i|_p$ and $Z_i(0)=\nabla^2_XZ_i(0)=0$ we deduce
\begin{equation*}
    \begin{split}
        \left.\frac{d^2}{dt^2} h_{ij}(t)\right|_{t=0}&=2\omega_p(\nabla_XZ_i(0),\nabla_XZ_j(0))=2\omega_{12}(p)\\
        \left.\frac{d^3}{dt^3} h_{ij}(t)\right|_{t=0}&=0\\
        \left.\frac{d^2}{dt^2} h_{1i}(t)\right|_{t=0}&=0.
    \end{split}
\end{equation*}
Moving on to higher derivatives, for $r\geq 2$, $2\leq i<j$ we have
\begin{equation}\label{eq:Derhij}
\begin{split}
    &\left.\frac{d^{r+2}}{dt^{r+2}}\right|_{t=0}h_{ij}(t)=\sum_{q=0}^{r+2}\binom{r+2}{q}\omega_p(\nabla^{r+2-q}_XZ_i(0),\nabla^q_XZ_j(0))\\    &=(r+2)\omega_p(Q^r_0\partial_i|_p,\partial_j|_p)+(r+2)\omega_p(\partial_i|_p,Q^r_0\partial_j|_p)\\
    &+\sum_{q=2}^{r-4}\binom{r+2}{q+1}\omega_p(Q^{r-q}_0\partial_i|_p,Q^q_0\partial_j|_p)
\end{split}
\end{equation}
as well as
\begin{equation}\label{eq:Derh1i}
    \left.\frac{d^{r+1}}{dt^{r+1}}h_{1i}(t)\right|_{t=0}=\omega_p(X_p,\nabla^{r+1}_XZ_i(0))=\omega_p(\partial_1|_p,Q^r_0\partial_i|_p)
\end{equation}
where we use $\nabla_XX=\nabla_X\dot{\gamma}=0$. Now, as we saw earlier, $s^*_p(\omega|_U)=\omega|_U$ if and only if $\omega_{ij}\circ s_p=\omega_{ij}$ for all $1\leq i<j\leq 2n$, and since $(s_p\circ\gamma)(t)=\gamma(-t)$ (by definition of the geodesic symmetry centered at $p$) we see that $s^*_p(\omega|_U)=\omega|_U$ if and only if the $h_{ij}$ are even and the $h_{1i}$ odd for any choice of $X_p\in T_pM-\{0\}$, thus we have the following:
\begin{lemma}\label{lemma:Derh}
    If $s^*_p(\omega|_U)=\omega|_U$ then for any choice of $X_p\in T_pM-\{0\}$, $\forall 2\leq i<j\leq 2n$ the odd derivatives of $h_{ij}$ vanish and the even derivatives of $h_{1i}$ vanish at $0$.

    Conversely, if the connection $\nabla$ is analytic while the odd derivatives of $h_{ij}$ and the even derivatives of $h_{1i}$ vanish at $0$ for all $2\leq i<j\leq 2n$ and all choices of $X_p\in T_pM-\{0\}$, then there exists a symmetric normal neighborhood $V\subseteq U$ of $p$ such that $s^*_p(\omega|_V)=\omega|_V$.
\end{lemma}

As we saw, the odd derivatives of the $h_{ij}$ always vanish at 0 up to order 3 and the same goes for the even derivatives of the $h_{1i}$ up to order 2, so in view of equations \ref{eq:Derhij} and \ref{eq:Derh1i}, the conditions on the derivatives of $h_{ij}$ and $h_{1i}$ stated in lemma \ref{lemma:Derh} become
\begin{equation*}
    \begin{split}
        0=\left.\frac{d^{r+2}}{dt^{r+2}}h_{ij}(t)\right|_{t=0}=&(r+2)\omega_p(Q^r_0\partial_i|_p,\partial_j|_p)+(r+2)\omega_p(\partial_i|_p,Q^r_0\partial_j|_p)\\
    &+\sum_{q=2}^{r-4}\binom{r+2}{q+1}\omega_p(Q^{r-q}_0\partial_i|_p,Q^q_0\partial_j|_p)\\
        0=\left.\frac{d^{r+1}}{dt^{r+1}}h_{1i}(t)\right|_{t=0}=&\omega_p(\partial_1|_p,Q^r_0\partial_i|_p).
    \end{split}
\end{equation*}
for $r=3,5,7,...$, $2\leq i<j\leq 2n$ and any choice of $X_p\in T_pM-\{0\}$. The equations associated to the $h_{1i}$ can actually be put in the same form as the others. Indeed, the skew-symmetry of the curvature tensor yields the equation
\begin{equation*}
    \Pi X=R(X,X)X=0
\end{equation*}
to which we can apply $\nabla_X$ to obtain
\begin{equation*}
    (\nabla^l_X\Pi)X=0
\end{equation*}
for all $l\geq 0$, from which we deduce
\begin{equation*}
    Q^q_0\partial_1|_p=0
\end{equation*}
for all $q\geq 2$. Thus for $r=3,5,7,...$, $0=\left.\frac{d^{r+1}}{dt^{r+1}}h_{1i}(t)\right|_{t=0}$ is equivalent to 
\begin{equation*}
    0=(r+2)\omega_p(Q^r_0\partial_1|_p,\partial_i|_p)+(r+2)\omega_p(\partial_1|_p,Q^r_0\partial_i|_p)+\sum_{q=2}^{r-4}\binom{r+2}{q+1}\omega_p(Q^{r-q}_0\partial_1|_p,Q^q_0\partial_i|_p).
\end{equation*}
Finally, as $(\partial_i|_p)_{i=1}^{2n}$ forms a basis of $T_pM$ we arrive at the following result:
\begin{theorem}\label{theorem:Qr}
    For $r=3,5,7,...$ and $p\in M$ we'll say that the Fedosov manifold $(M,\omega,\nabla)$ satisfies the (Q$r$) condition at 
    $p$ if for any choice of $X_p\in T_pM-\{0\}$ we have
    \begin{equation}\label{eq:Qr}
        0=(r+2)\omega_p(Q^r_0-,-)+(r+2)\omega_p(-,Q^r_0-)+\sum_{q=2}^{r-4}\binom{r+2}{q+1}\omega_p(Q^{r-q}_0-,Q^q_0-).
    \end{equation}
    We'll say that $(M,\omega,\nabla)$ satisfies the condition (Q$r$) if it satisfies (Q$r$) at every point.

    If $\nabla$ is of type S, then $(M,\omega,\nabla)$ satisfies (Q$r$) for all $r=3,5,7,...$. Conversely if $\nabla$ is analytic and while $(M,\omega,\nabla)$ satisfies (Q$r$) for all $r=3,5,7,...$ then $\nabla$ is of type S.
\end{theorem}

The conditions (Q$r$) admit some alternative forms, which might be more convenient depending on the context. Firstly, for $r=3,5,7,...$ we can introduce the endomorphism field $P^r$ along $\gamma$
\begin{equation}
    P^r:=(r+2)Q^r+\sum_{q=2}^{(r-1)/2}\binom{r+2}{q+1}(Q^q)^\top\circ Q^{r-q}
\end{equation}
where $(-)^\top$ denotes the transpose with respect to the symplectic form, i.e. for $z\in M$ and $A_z\in\text{End}(T_zM)$, $A^\top_z$ is defined by
\begin{equation*}
    \omega_z(-,A_z-)=\omega(A^\top_z-,-).
\end{equation*}
Then 
\begin{equation*}
    \omega_p(P^r_0-,-)+\omega_p(-,P^r_0)=(r+2)\omega_p(Q^r_0-,-)+(r+2)\omega_p(-,Q^r_0-)+\sum_{q=2}^{r-4}\binom{r+2}{q+1}\omega_p(Q^{r-q}_0-,Q^q_0-)
\end{equation*}
so $(M,\omega,\nabla)$ satisfies (Q$r$) at $p$ if and only if for any choice of $X_p\in T_pM-\{0\}$, $P^r_0$ belongs to $\mathfrak{sp}(T_pM,\omega_p)$, which is equivalent to $(P^r_0)^\top=-P^r_0$.

Secondly, using $\nabla_XX=0$ and reasoning inductively from the definition $\Pi:=R(X,-)X$, we can reexpress $\nabla^l_X\Pi$ as follows:
\begin{proposition}\hfill
    \begin{enumerate}[(i)]
        \item For $l\geq 0$, $\nabla^l_X\Pi=(\nabla^l_XR)(X,-)X=(\nabla^lR)(X,-;X,...,X)X$.
        \item In particular for $l\geq 0$, $\nabla^l_X\Pi_0=(\nabla^lR)_p(X_p,-;X_p,...,X_p)X_p$.
    \end{enumerate}
\end{proposition}
By injecting (ii) in equation \ref{eq:CoeffQ}, we can then obtain a version of condition (Q$r$) which does not require a choice of geodesic $\gamma$, but only involves the tangent vector $X_p\in T_pM$, the covariant derivatives of $R$ and the symplectic form $\omega$.

Notably, the first condition, (Q3), takes a particularly simple form:
\begin{theorem}
    The Fedosov manifold $(M,\omega,\nabla)$ satisfies (Q3) at a point $p\in M$ if and only if for all $X_p\in T_pM$, $(\nabla_{X_p}R)(X_p,-)X_p$ belongs to $\mathfrak{sp}(T_pM,\omega_p)$.
\end{theorem}

As a corollary, we see that a connection of type S is necessarily preferred
\begin{corollary}
    If a symplectic connection $\nabla$ is of type $S$, then $\nabla$ is preferred (i.e. has cyclic parallel Ricci curvature).
\end{corollary}
\begin{proof}
    Elements of the Lie algebra of the symplectic group have vanishing trace, so for any vector field $X\in\mathfrak{X}(M)$,
    \begin{equation*}
        0=\text{tr}(\nabla_XR)(X,-)X=(\nabla_Xr)(X,X).
    \end{equation*}
\end{proof}

\section{The Sekigawa-Vanhecke theorem and preferred connections that are not of type S}

The original impetus behind our work came from a theorem of Sekigawa and Vanhecke, on almost Hermitian manifolds. Relying on this theorem, we'll also be able to show that the inclusion "S-type" $\subset$ "preferred" derived at the end of the last section is a strict one.

\begin{theorem}[Sekigawa-Vanhecke, 1986 \cite{SkVh:1986}]\label{thm:SkVh}
    Let $(M,g,J)$ be an almost hermitian manifold whose local geodesic symmetries preserve the Kähler form $\omega=g(-,J-)$. Then $(M,J,g)$ is a locally symmetric hermitian space. In particular $(M,J,g)$ is Kähler with a canonical symplectic Kähler form $\omega=g(-,J-)$, which is parallel for its Levi-Civita connection.
\end{theorem}

In the book of Besse \cite{Besse:Einstein} (theorem 14.98, chapter 14, page 420), the reader will find a list (concocted by D.V. Alekseevskii)  of Quaternion-Kähler manifolds which are not locally symmetric. Such a manifold is always Kähler-Einstein, thus its Levi-Civita connection is necessarily Ricci-parallel, and in particular preferred (as a symplectic connection for the Kähler form). 

On the other hand, in view of the Sekigawa-Vanhecke theorem, the Levi-Civita connection of a non-locally-symmetric Kähler manifold cannot be of type S. And so we see that this list provides examples of preferred symplectic connections which are not S-type.

\begin{corollary}
    There exist preferred connections which are not of type S.
\end{corollary}

\section{Ricci-type connections}

\subsection{Recalls and preliminaries}

Locally symmetric symplectic spaces provide an obvious class of S-type symplectic connections. In this section we'll show that Ricci-type connections constitute another, less obvious class of S-type connections. In particular this will allow us to exhibit many examples of non-locally-symmetric S-type connections. 

To motivate the definition of a symplectic connection of Ricci-type, we'll begin by presenting the symmetries exhibited by the curvature tensor of a symplectic connection in a new light. Let $(M,\omega,\nabla)$ be a Fedosov manifold and choose a point $x\in M$. That the covariant curvature tensor $\underline{R}_x$ is skew-symmetric in its first two arguments and symmetric in its last two arguments can be expressed concisely by saying that $\underline{R}_x$ is an element of $\Lambda^2 T^*_xM\otimes S^2T^*_xM$, where $S^2T^*_xM$ denotes the symmetrization of $T^*_xM\otimes T^*_xM$. The fact that $\underline{R}_x$ satisfies the first Bianchi identity $\mathfrak{S}_{X,Y,Z}\underline{R}_x(X,Y,Z,W)=0$ is equivalent to $\underline{R}_x\in \mathcal{R}_x:=\ker a$, where $a$ is the skewsymmetrization map
\begin{equation*}
    \begin{split}
        a:\Lambda^pT^*_xM\otimes S^qT^*_xM&\to\Lambda^{p+1}T^*_xM\otimes S^{q-1}T^*_xM\\
        a(v^1\wedge...\wedge v^p\otimes w^1...w^q)&:=\sum_{i=1}^qv^1\wedge...\wedge v^p\wedge w^i\otimes w^1...\hat{w}^i...w^q.
    \end{split}
\end{equation*}
\begin{remark}
    For $q\geq 0$, the skewsymmetrization maps assemble into the Koszul long exact sequence
    \begin{equation*}
        0\longrightarrow S^qV\overset{a}{\longrightarrow}S^{q-1}\otimes V\overset{a}{\longrightarrow}S^{q-2}\otimes \Lambda^2V\overset{a}{\longrightarrow}...\overset{a}{\longrightarrow}\Lambda^qV\longrightarrow 0
    \end{equation*}
    for any finite-dimensional vector space $V$. One can also obtain a long exact sequence going in the opposite direction using symmetrization maps (see \cite{BieliavskyEtAl:2005}).
\end{remark}
Thus $\mathcal{R}_x$ is the space of 4-tensors satisfying the algebraic identities of a symplectic curvature tensor. The symplectic group $\text{Sp}(T_xM,\omega_x)$ acts on $\mathcal{R}_x$, which splits into two irreducible components when $\dim M=2n\geq 4$ so that
\begin{equation*}
    \mathcal{R}_x=\mathcal{E}_x\oplus\mathcal{W}_x.
\end{equation*}
If we decompose the covariant symplectic curvature tensor $\underline{R}_x$ into its $\mathcal{E}_x$ and $\mathcal{W}_x$ components (see \cite{Vaisman:1985}):
\begin{equation*}
    \underline{R}_x=\underline{E}_x+\underline{W}_x
\end{equation*}
then $\underline{E}_x$ is given by
\begin{equation*}
\begin{split}
    \underline{E}_x(X,Y,Z,T)=&-\frac 1{2(n+1)}\left[2\omega_x(X,Y)r_x(Z,T)+\omega_x(X,Z)r_x(Y,T)\right.\\
    &+\left.\omega_x(X,T)r(Y,Z)-\omega_x(Y,Z)r_x(X,T)-\omega_x(Y,T)r_x(X,Z)\right].
\end{split}
\end{equation*}
Correspondingly, the $(1,3)$ curvature tensor can be written $R_x=E_x+W_x$ with
\begin{equation}\label{eq:RicciTypeCurv}
\begin{split}
    E(X,Y)Z=&\frac 1{2(n+1)}\left[2\omega(X,Y)\rho Z+\omega(X,Z)\rho Y-\omega(Y,Z)\rho X\right.\\
    &\left.+\omega(X,\rho Z)Y-\omega(Y,\rho Z)X\right]
\end{split}
\end{equation}

\begin{definition}
    Let $(M,\omega,\nabla)$ be a Fedosov manifold of dimension $2n\geq 4$. We say that $\nabla$ is of \textbf{Ricci-type} if $W=0$, or equivalently if $R=E$.
\end{definition}
\begin{remark}
    The 2-dimensional case is peculiar (for example, the curvature of a symplectic connection over a symplectic surface is \emph{always} determined by its Ricci curvature) and is treated separately in a later section.
\end{remark}

One striking property of Ricci-type connections is the existence of a vector field $U$ and a real number $K$ which, when combined with $\rho$, essentially control the local geometry (see \cite{Cahen:RicciType}):
\begin{lemma}\label{lemma:PropRicciType}
    Let $(M,\omega,\nabla)$ be a Fedosov manifold of dimension $2n\geq 4$ with $\nabla$ of Ricci-type. For a vector field $X$ on $M$, we write $\underline{X}:=\iota_X\omega=\omega(X,-)$. Then
    \begin{enumerate}[(i)]
        \item There exists a vector field $U$ such that
        \begin{equation*}
            \nabla_X\rho=-\frac 1{2n+1}(X\otimes\underline{U}+U\otimes\underline{X});
        \end{equation*}
        \item There exists a function $f$ such that
        \begin{equation*}
            \nabla_XU=-\frac{2n+1}{2(n+1)}(\rho^2)X+fX;
        \end{equation*}
        \item There exists a real number $K$ such that
        \begin{equation*}
            \text{tr}(\rho^2)+\frac{4(n+1)}{2n+1}f=K.
        \end{equation*}
    \end{enumerate}
\end{lemma}
\begin{corollary}\label{cor:RicciTypePreferred}
    Any Ricci-type connection is preferred (once again, this holds in dimension $\geq 4$).
\end{corollary}

Taking a moment to collect the different facts concerning Ricci-type connections we've mentioned so far, we see that given a Ricci-type connection $\nabla$, its curvature $R$ is determined by $\rho$ (i.e. by its Ricci curvature); thus its covariant derivative $\nabla R$ is determined by $\nabla\rho$, which is itself determined by the vector field $U$ ((i) in lemma \ref{lemma:PropRicciType}). Consequently, $\nabla^2 R$ is determined by $\nabla U$, which is in turn determined by $\rho$ and $f$ ((ii) in \ref{lemma:PropRicciType}). As $f$ itself is determined by the real number $K$ and $\rho$ ((iii) in lemma \ref{lemma:PropRicciType}), we see that all covariant derivatives of $R$ can be expressed in terms of $U$, $\rho$ and $K$. 

If to this we add the fact that a Ricci-type connection is always analytic (see \cite{BieliavskyEtAl:2005}).
\begin{proposition}
    Any Ricci-type connection is analytic.
\end{proposition}
We can indeed conclude that for a point $x\in M$, $\omega_x$, $\rho_x$, $U_x$ and $K$ determine the affine and symplectic structure of $(M,\omega,\nabla)$ in neighborhood of $x$ (see \cite{KN:FundDiffGeoI} theorem 7.2 and corollary 7.3).

\subsection{Construction of Ricci-type connections by reduction}

Here we outline a reduction procedure which yields a Ricci-type Fedosov manifold as the Marsden-Weinstein quotient of a hypersurface (defined by a quadratic equation) in a flat symplectic space, first presented by Baguis and Cahen in \cite{Baguis:2000}. 

Let $\Omega'$ be the standard flat symplectic form on $\mathbb{R}^{2n+2}$, and let $A$ be a non-zero element of $\mathfrak{sp}(\mathbb{R}^{2n+2},\Omega')$. The action of the 1-parameter subgroup $\{\exp(tA)\}$ of the symplectic group on $\mathbb{R}^{2n+2}$ is Hamiltonian, more precisely we have
\begin{equation*}
    \iota(A^*)\Omega'=dH_A
\end{equation*}
where $H_A(x):=\frac 12\Omega'(x,Ax)$ and $A^*$ is the corresponding fundamental vector field, defined by $A^*_x:=\left.\frac d{dt}e^{-tA}x\right|_{t=0}=-Ax$.

We consider a level set of the Hamiltonian:
\begin{equation*}
    \Sigma_A=\{x\in\mathbb{R}^{2n+2}|\Omega'(x,Ax)=1\}
\end{equation*}
(replace $A$ by $-A$ if necessary to ensure that $\Sigma_A$ is non-empty), a hypersurface in $\mathbb{R}^{2n+2}$. $\Sigma_A$ is invariant under the action of the 1-parameter subgroup $\{e^{tA}|t\in\mathbb{R}\}$ so that we can construct its Marsden-Weinstein quotient 
\begin{equation*}
    \pi:\Sigma_A\to M:=\Sigma_A/\{e^{tA}|t\in\mathbb{R}\}.
\end{equation*} 

For $x\in\Sigma_A$ we have $T_x\Sigma_A=\rangle Ax\langle^\bot$, where $\rangle v_1,...,v_p\langle$ denotes the subspace spanned by $v_1,...,v_p$ and $\bot$ denotes the orthogonal with respect to $\Omega'$. We define the horizontal subspace $\mathcal{H}_x:=\rangle x,Ax\langle^\bot$. The differential $\pi_{*x}$ then restricts to an isomorphism between $\mathcal{H}_x$ and the tangent space $T_{y=\pi(x)}M$, thus for any vector $X_y\in T_yM$ there exists a unique vector $\overline{X}_x\in \mathcal{H}_x$ such that $X_y=\pi_{*x}\overline{X}_x$, the horizontal lift of $X_y$ at $x$.

Over $M$ we define the reduced symplectic form
\begin{equation*}
    \omega_{y=\pi(x)}(X_y,Y_y)=\Omega'(\overline{X}_x,\overline{Y}_x).
\end{equation*}
And if we let $\nabla'$ denotes the standard flat connection on $\mathbb{R}^{2n+2}$, then the formula
\begin{equation*}
    (\nabla_XY)_{y=\pi(x)}:=\pi_{*x}((\nabla'_{\overline{X}}\overline{Y})_x-\Omega'(A\overline{X}_x,\overline{Y}_x)x+\Omega'(\overline{X}_x,\overline{Y}_x)Ax)
\end{equation*}
defines an affine connection on $M$.

\begin{theorem}[Baguis-Cahen, 2000 \cite{Baguis:2000}]
    Assume $2n\geq 4$, then $\nabla$ is a Ricci-type connection on $(M,\omega)$.
\end{theorem}

A simple criterium allows us to determine whether the manifold obtained through reduction is locally symmetric.

\begin{theorem}[Cahen-Gutt-Schwachhöfer, 2003]
    $(M,\omega,\nabla)$ is locally symmetric if and only if $A^2=\lambda I$ for a constant $\lambda\in\mathbb{R}$.
\end{theorem}

In particular, we can exhibit non-locally-symmetric Ricci-type connections.

\begin{corollary}
    Let $2n\geq 4$, there exist non-locally-symmetric Ricci-type connections of dimension $2n$.
\end{corollary}
\begin{proof}
    For any $2n\geq 4$, one can find infinitely many $A\in\mathfrak{sp}(\mathbb{R}^{2n+2},\Omega')-\{0\}$ for which $A^2$ is not proportional to the identity. For example one could take
    \begin{equation*}
        A:=\begin{pmatrix}
            E & I\\
            I & E
        \end{pmatrix}
    \end{equation*}
    with $E$ a non-zero antisymmetric $(n+1)\times (n+1)$ matrix.
\end{proof}

\subsection{Ricci-type connections are of type S}

Investigating the form that the (Q$r$) conditions will take for Ricci-type connections, and more generally determining which Ricci-type connections are of type S, are questions that one would naturally be lead to given the considerations presented in the previous section. The answer turns out to be particularly simple: a Ricci-type connection will satisfy all (Q$r$) conditions (for $r=3,5,7,...$) and thus be of type S, given that Ricci-type connections are always analytic.
\begin{theorem}\label{thm:RicciTypeQrCdts}
    Let $(M,\omega,\nabla)$ be a Fedosov manifold of dimension $2n\geq 4$, with $\nabla$ of Ricci-type. Then $\nabla$ satisfies all (Q$r$) conditions for $r=3,5,7,...$.
\end{theorem}
\begin{corollary}
    Any Ricci-type connection is of type S.
\end{corollary}

Before proceeding with the proof, we bring the reader's attention to two easy consequences of our last finding. The first we obtain by invoking the Sekigawa theorem, the second by recalling the existence of non-locally-symmetric Ricci-type connections.

\begin{corollary}
    If the Levi-Civita connection of a Kähler manifold $(M,J,g)$ (of dimension $\geq 4$) is of Ricci-type (for the Kähler form), then $(M,J,g)$ is a locally symmetric Hermitian space.
\end{corollary}

\begin{corollary}
    There exist non-locally-symmetric S-type Fedosov manifold of any dimension $2n\geq 4$.
\end{corollary}

\begin{proof}[Proof of theorem \ref{thm:RicciTypeQrCdts}]
    
Let $(M,\omega,\nabla)$ be a Fedosov manifold of dimension $2n\geq 4$ with $\nabla$ of Ricci-type, let us fix a point $p\in M$, a non-zero tangent vector $X_p\in T_pM-\{0\}$, and let us define the endomorphism fields $\Pi$, $Q^r$ and $P^r$ along the geodesic $\gamma$ (defined by $\dot{\gamma}(0)=X_p$) as in section \ref{sect:RecCdt}.

Our first step will be reexpressing the covariant derivatives of $\Pi$ in terms of the Ricci endomorphism $\rho$ as well as the vector field $U$, the function $f$ and the constant $K$ that were introduced in lemma \ref{lemma:PropRicciType}. Making use of equation \ref{eq:RicciTypeCurv} (expressing the curvature of a Ricci-type connection in terms of $\omega$ and $\rho$), $\Pi:=R(X,-)X$ can be rewritten
\begin{equation}\label{eq:PiRicciType}
    \begin{split}
        \Pi=&\frac 1{2(n+1)}\left(3\omega(X,-)\rho X-r(X,-)X+r(X,X)\right).
    \end{split}
\end{equation}

Using the fact that a Ricci-type connection is preferred (corollary \ref{cor:RicciTypePreferred}), which implies $(\nabla_Xr)(X,X)=0$, we compute
\begin{equation}\label{eq:RicciTypeDPi1}
        \nabla_X\Pi=\frac 1{2(n+1)}\left(3\omega(X,-)(\nabla_X\rho)X-(\nabla_Xr)(X,-)X\right).
\end{equation}
From point (i) in lemma \ref{lemma:PropRicciType} we obtain
\begin{equation}\label{eq:Drho}
    (\nabla_X\rho)X=\frac 1{2n+1}\omega(X,U)X
\end{equation}
\begin{equation}\label{eq:Dr}
    (\nabla_Xr)(X,-)=\omega(X,\nabla_X\rho-)=\frac 1{2n+1}\omega(U,X)\omega(X,-)
\end{equation}
and injecting \ref{eq:Drho} and \ref{eq:Dr} in \ref{eq:RicciTypeDPi1} we get
\begin{equation}\label{eq:RicciTypeDPi}
    \nabla_X\Pi=\frac 4{(n+1)(2n+1)}\omega(-,X)\omega(U,X)X.
\end{equation}

Applying $\nabla_X$ to \ref{eq:RicciTypeDPi}, we get
\begin{equation}\label{eq:RicciTypeD2Pi1}
    \nabla^2_X\Pi=\frac 4{(n+1)(2n+1)}\omega(-,X)\omega(\nabla_XU,X)X.
\end{equation}
And since $\omega(\rho^2X,X)=-\omega(\rho X,\rho X)=0$ (as $\rho$ is a symplectic endomorphism), from point (ii) in lemma \ref{lemma:PropRicciType}
\begin{equation*}
    \nabla_XU=-\frac {2n+1}{2(n+1)}(\rho^2)X+fX
\end{equation*}
we deduce
\begin{equation*}
    \omega(\nabla_XU,X)=0
\end{equation*}
which we can inject in \ref{eq:RicciTypeD2Pi1} to obtain
\begin{equation}\label{eq:RicciTypeD2Pi}
    \nabla^2_X\Pi=0.
\end{equation}
Thus we encounter our first major simplification: for a Ricci-type connection, $\nabla^q_X\Pi$ vanishes for all $q\geq 2$. In particular, the only terms that will remain in the expressions of the endomorphisms fields $Q^r$ $r\geq 2$ (we refer to the expressions given in equation \ref{eq:CoeffQ}) will be those obtained by composing $\Pi$ and $\nabla_X\Pi$ (i.e. those for which $i_j\leq 1$ for all $j$).

We can further simplify the expressions of the $Q^r$ by examining "binary" compositions of $\Pi$ and $\nabla_X\Pi$. From the expression for $\nabla_X\Pi$ given in \ref{eq:RicciTypeDPi} (and $\omega(X,X)=0$) we obtain
\begin{equation}\label{eq:RicciTypeDPiDPi}
    \nabla_X\Pi\circ\nabla_X\Pi=0.
\end{equation}
Putting together equations \ref{eq:RicciTypeDPi} and \ref{eq:PiRicciType} we can also compute
\begin{equation}\label{eq:RicciTypePiDPi}
    \begin{split}
        \Pi\circ\nabla_X\Pi=&\frac 2{2(n+1)^2(2n+1)}\left[-\omega(-,X)\omega(U,X)r(X,X)X\right.\\
        &\left.+\omega(-,X)\omega(U,X)r(X,X)X\right]\\
        =&0
    \end{split}
\end{equation}
and
\begin{equation}\label{eq:RicciTypeDPiPi}
    \begin{split}
        \nabla_X\Pi\circ\Pi=&\frac 2{(n+1)^2(2n+1)}\omega(U,X)X\left[3\omega(X,-)\underbrace{\omega(\rho X,X)}_{-r(X,X)}+r(X,X)\omega(-,X)\right]\\
        &=\frac 8{(n+1)^2(2n+1)}r(X,X)\omega(U,X)\omega(-,X)X\\
        &=\frac 2{n+1}r(X,X)\nabla_X\Pi.
    \end{split}
\end{equation}
Now for $r\geq 2$, recalling equation \ref{eq:CoeffQ} for the expression of $Q^r$:
\begin{equation*}
    Q^r=\sum c^r_{i_1...i_k}\nabla^{i_1}_X\Pi\circ...\circ\nabla^{i_k}_X\Pi
\end{equation*}
where we sum over $1\leq k$, and $0\leq i_j$ with $r=2k+i_1+...+i_k$, if we combine equations \ref{eq:RicciTypeD2Pi}, \ref{eq:RicciTypeDPiDPi} and \ref{eq:RicciTypePiDPi}, we see that $\nabla^{i_1}_X\Pi\circ...\circ\nabla^{i_k}_X\Pi$ can only be nonzero when $i_2=i_3=...=i_k=0$ and $i_1\leq 1$. If we assume that $r=2k+i_1+...+i_k$ is odd, then $i_1$ \emph{must} be equal to $1$, and using equation \ref{eq:RicciTypeDPiPi} we have
\begin{equation*}
    \nabla^{i_1}_X\Pi\circ...\circ\nabla^{i_k}_X\Pi=\nabla_X\Pi\circ\Pi\circ...\circ\Pi=\left(\frac 2{n+1}r(X,X)\right)^{k-1}\nabla_X\Pi
\end{equation*}
with $k=\frac{r-1}2$. In conclusion, for $r=3,5,7,...$, $Q^r$ is a multiple of $\nabla_X\Pi$, more precisely we have
\begin{equation}\label{eq:RicciTypeQrOdd}
    Q^r=B_rr(X,X)^{(r-3)/2}\nabla_X\Pi
\end{equation}
for some constant $B_r\in\mathbb{R}$ which only depends on $r$.

Using \ref{eq:RicciTypeDPi} once again, we can check that $\nabla_X\Pi$ is always a symplectic endomorphism (i.e. that $\nabla_X\Pi^\top=-\nabla_X\Pi$). Indeed we have
\begin{equation*}
    \begin{split}
        \omega(\nabla_X\Pi Y,Z)+\omega(Y,\nabla_X\Pi Z)&=\frac 2{(n+1)(2n+1)}\omega(U,X)\omega(Y,X)\left(\omega(X,Z)+\omega(Z,X)\right)\\
        &=0
    \end{split}
\end{equation*}
which shows in particular that a Ricci-type connection always satisfies the (Q3) condition. For $r$ odd with $r\geq 5$, recall that (Q$r$) is satisfied if and only if the endomorphism field
\begin{equation*}
    P^r:=(r+2)Q^r+\sum_{q=2}^{(r-1)/2}\binom{r+2}{q+1}(Q^q)^\top\circ Q^{r-q}
\end{equation*}
is symplectic (for all choices of $p\in M$ and $X_p\in T_pM-\{0\}$). 

We already know that $Q^r$ is symplectic as a multiple of $\nabla_X\Pi$ (equation \ref{eq:RicciTypeQrOdd}). We'll finish the proof by showing that $(Q^q)^\top\circ Q^{r-q}$ is symplectic for all $2\leq q\leq (r-1)/2$. Here, $r$ is odd so either $q$ is odd or $(r-q)$ is odd, and since $(Q^q)^\top\circ Q^{r-q}$ is symplectic if and only if its transpose $\left[(Q^q)^\top\circ Q^{r-q}\right]^\top=(Q^{r-q})^\top\circ Q^q$ is symplectic, we can assume without loss of generality that $q$ is odd.

In that case $(Q^q)^\top=Q^q=B_qr(X,X)^{(q-3)/2}\nabla_X\Pi$ and we can once again use equations \ref{eq:RicciTypeDPiDPi} and \ref{eq:RicciTypePiDPi} to show that
\begin{equation*}
    (Q^q)^\top\circ Q^{r-q}=\lambda\nabla_X\Pi\circ Q^{r-q}=\mu \nabla_X\Pi
\end{equation*}
for some real-valued functions $\lambda$, $\mu$. In particular $(Q^q)^\top\circ Q^{r-q}$ must be symplectic.

\end{proof}

\subsection{Homogeneous Ricci-type and non-Ricci-type symmetric spaces}

The class of compact homogeneous Ricci-type connections was investigated in \cite{CahenGutt:HomRicci}, where, after imposing conditions on the fundamental group, the following was obtained:

\begin{theorem}
    Let $(M,\omega)$ be a compact homogeneous symplectic manifold with finite fundamental group. If $(M,\omega)$ admits a homogeneous symplectic connection $\nabla$ with Ricci-type curvature, then $(M,\omega)$ is symplectomorphic to $(\mathbb{P}_n(\mathbb{C}),\omega_0)$ where $\omega_0$ is a multiple of the Kähler form of the Fubini-Study metric, and $\nabla$ is affinely equivalent to the Levi-Civita connection.
\end{theorem}

As there exist many examples of compact, simply-connected, Hermitian symmetric spaces \emph{other} than the complex projective spaces, we see that the class of S-type connections is strictly larger than that of Ricci-type connections.

\begin{corollary}
    There exist locally symmetric symplectic connections (in particular, S-type connections) which are not of Ricci-type.
\end{corollary}

\section{Preferred connections over symplectic surfaces}

The special properties of the 2-dimensional case force us to treat it separately. The curvature of a symplectic connection over a symplectic surface will \emph{always} be determined by its Ricci curvature - a property they share with Ricci-type connections - but contrary to Ricci-type connections, not all of them are preferred. As a counter-example, one can take the Levi-Civita connection of any non-locally-symmetric Kähler surface since we have the following:

\begin{proposition}
    Let $(M,g,J)$ be a Kähler surface. If its Levi-Civita connection is preferred (for the Kähler form), then $(M,g,J)$ is a locally symmetric Hermitian space. 
\end{proposition}

Thus if we aim to find an analogue of theorem \ref{thm:RicciTypeQrCdts} in the 2-dimensional setting, it seems that we should consider preferred symplectic connections over symplectic surfaces. The following result suggests that this is indeed the right approach:
\begin{theorem}\label{thm:PreferredDim2QrCdts}
    Let $(M,\omega,\nabla)$ be a Fedosov manifold of dimension two. If $\nabla$ is preferred, then $(M,\omega,\nabla)$ satisfies the (Q$r$) conditions for all $r=3,5,7,...$.
\end{theorem}
\begin{corollary}\label{cor:TypeSPreferredAnalyticSurface}
    Let $\nabla$ be an analytic preferred  connection over a symplectic surface. Then $\nabla$ is of type S.
\end{corollary}

This happens in part because the conditions (Q$r$) become much simpler (as $h_{12}$ is the only $h_{ij}$ function we must consider), and in part because, as we already mentioned, the curvature tensor $R$ of a symplectic connection over a surface is determined by its Ricci curvature $r$. We begin with the latter statement.

\begin{theorem}
    Let $(M,\omega,\nabla)$ be a Fedosov manifold, with $M$ of dimension 2. Then
    \begin{equation*}
        \underline{R}=-\omega\otimes r
    \end{equation*}
    or equivalently
    \begin{equation*}
        R=\omega\otimes\rho
    \end{equation*}
    where $\rho\in\Gamma(\text{End}(TM))$ is the (symplectic) Ricci endomorphism defined by $\omega(-,\rho-):=r$.
\end{theorem}
\begin{proof}
    Since $M$ is 2-dimensional, the symplectic form $\omega$ is a volume form on $M$ and we have
    \begin{equation*}
        \begin{split}
            r(Z,T)\omega(X,Y)&=tr(R(Z,-)T)\omega(X,Y)=\omega(R(Z,X)T,Y)+\omega(X,R(Z,Y)T)\\
            &=\underline{R}(Z,X,T,Y)-\underline{R}(Z,Y,T,X)=\underline{R}(Z,X,Y,T)+\underline{R}(Y,Z,X,T)\\
            &=-\underline{R}(X,Y,Z,T)
        \end{split}
    \end{equation*}
    making use of the symmetries of $\underline{R}$. In particular, Bianchi's first identity in used in the last equality.
\end{proof}

Using the same notation as in section \ref{sect:RecCdt}, in dimension 2 we need only consider the component $\omega_{12}$ of the symplectic form, so $s^*_p(\omega|_U)=\omega|_U$ if and only if $h_{12}(t)$ is odd, and for $r=3,5,7,...$ the (Q$r$) condition becomes
\begin{equation*}
\begin{split}
    0=&\left.\frac{d^{r+1}}{dt^{r+1}}h_{12}(t)\right|_{t=0}=\omega_p(\partial_1|_p,Q^r_0\partial_2|_p)\\
    =&\omega_p(Q^r_0\partial_1|_p,\partial_2|_p)+\omega_p(\partial_1|_p,Q^r_0\partial_2|_p)\\
    =&(\text{tr}Q^r_0)\omega_{12}(p).
\end{split}
\end{equation*}
Thus $(M,\omega,\nabla)$ satisfies the (Q$r$) condition at $p$ if and only if
\begin{equation}\label{eq:Dim2Qr}
    0=\text{tr}Q^r_0
\end{equation}
for any choice of $X_p\in T_pM-\{0\}$. Notice that (Q$3$) is equivalent to requiring that $\nabla$ be preferred. With these considerations in hand, we can prove the result announced at the beginning of this section.
\begin{proof}[Proof of theorem \ref{thm:PreferredDim2QrCdts}]
    The core observation is the following: for $q\geq1$ and $j\geq0$
    \begin{equation}\label{eq:NullPiComposition}
        \nabla^j_X\Pi\circ\nabla^q_X\Pi=0
    \end{equation}
    which we prove in four steps:
    \begin{enumerate}
        \item Using $R=\omega\otimes\rho$, $\nabla\omega=0$ and $\nabla_X\dot{\gamma}=0$, we get
        \begin{equation}\label{eq:ProvingNullPiComposition-1}
            \nabla^j_X\Pi=\omega(\dot{\gamma},-)(\nabla^j_X\rho)\dot{\gamma}
        \end{equation}
        for all $j\geq0$.
        \item By definition of a preferred connection, $(\nabla_Xr)(\dot{\gamma},\dot{\gamma})=0$. Combining this with $\nabla_X\dot{\gamma}=0$, we deduce
        \begin{equation}\label{eq:ProvingNullPiComposition-2}
            (\nabla^q_Xr)(\dot{\gamma},\dot{\gamma})=0           
        \end{equation}
        for all $q\geq 1$.
        \item Using $\nabla\omega=0$ once again, and $r=\omega(-,\rho-)$ we obtain
        \begin{equation}\label{eq:ProvingNullPiComposition-3}
           \nabla^j_Xr=\omega(-,(\nabla^j_X\rho)-) 
        \end{equation}
        for all $j\geq 0$.
        \item Let $q\geq 1$ and $j\geq 0$. Putting together \ref{eq:ProvingNullPiComposition-1}, \ref{eq:ProvingNullPiComposition-2} and \ref{eq:ProvingNullPiComposition-3} we calculate
        \begin{equation*}
            \begin{split}
                \nabla^j_X\Pi\circ\nabla^q_X\Pi&=\omega(\dot{\gamma},-)\nabla^j_X\Pi((\nabla^q_X\rho)\dot{\gamma})\\
                &=\omega(\dot{\gamma},-)\omega(\dot{\gamma},(\nabla^q_X\rho)\dot{\gamma})(\nabla^j_X\rho)\dot{\gamma}\\
                &=\omega(\dot{\gamma},-)\underbrace{(\nabla^q_Xr)(\dot{\gamma},\dot{\gamma})}_{=0}(\nabla^j_X\rho)\dot{\gamma}=0
            \end{split}
        \end{equation*}
    \end{enumerate}
    Furthermore, for $q\geq1$ we have
    \begin{equation}\label{eq:VanishingTracePiDerivative}
        tr\nabla^q_X\Pi=\frac{d^{q-1}}{dt^{q-1}}tr\nabla_X\Pi=\frac{d^{q-1}}{dt^{q-1}}(\nabla_Xr)(\dot{\gamma},\dot{\gamma})=0
    \end{equation}

    We are now in a position to show that $(M,\nabla)$ satisfies the (Q$r$) conditions. By definition, $\nabla$ satisfies (Q3) so we need only show that
    \begin{equation*}
        \text{tr}Q^r|_0=0
    \end{equation*}
    for $r\geq 5$. We start from formula \ref{eq:CoeffQ}, which we partially restate here for convenience
    \begin{equation*}
        Q^r=\sum c_{i_1...i_k}^r\nabla^{i_1}_X\Pi\circ...\circ\nabla^{i_k}_X\Pi
    \end{equation*}
    where the summation is done over $1\leq k$, $0\leq i_j$ such that $2k+i_1+...+i_j=r$. Fix $r$ odd with $r\geq 5$, $k\geq1$ and $i_j\geq0$ with $r=i_1+...+i_k+2k$. We will show that
    \begin{equation}\label{eq:VanishingTraceCompositionPi}
        \text{tr}(\nabla^{i_1}_X\Pi\circ...\circ\nabla^{i_k}_X\Pi)=0
    \end{equation}
    thus finishing the proof.

    \noindent \underline{Case 1:} if $k=1$ then $i_1=r-2\geq1$ and \ref{eq:VanishingTraceCompositionPi} becomes
    \begin{equation*}
        \text{tr}\nabla^{r-2}_X\Pi=0
    \end{equation*}
    which is a consequence of \ref{eq:VanishingTracePiDerivative}.

    \noindent \underline{Case 2:} if $k\geq2$, since $r$ is odd we cannot have $r=2k$ so there must exist $1\leq j\leq k$ such that $i_j\geq 1$. By the cyclic property of the trace, we can assume without loss of generality that $j=2$. By \ref{eq:NullPiComposition} we have $\nabla^{i_1}_X\Pi\circ\nabla^{i_2}_X\Pi=0$ so that
    \begin{equation*}
        \text{tr}(\nabla^{i_1}_X\Pi\circ\nabla^{i_2}_X\Pi\circ...\circ\nabla^{i_k}_X\Pi)=0.
    \end{equation*}
\end{proof}

\subsection{Non-symmetric S-type in dimension 2}

The family of (geodesically complete) non-locally-symmetric preferred symplectic connections on $\mathbb{R}^2$  constructed by Bourgeois and Cahen in section 11 of \cite{BourgeoisCahen:1998} constitutes a class of non-locally-symmetric symplectic connections of type S. This is an immediate consequence of corollary \ref{cor:TypeSPreferredAnalyticSurface} once one observes that these connections are analytic for the standard analytic structure on $\mathbb{R}^2$, and observing this poses no difficulty thanks to the very explicit formulae derived by Bourgeois and Cahen.

These explicit formulae are obtained by introducing a function $\beta$ and a 1-form $u$ which essentially control the geometry of a symplectic surface with preferred connection (the of $\beta$ and $u$ is another sign that such connections could be seen as the 2-dimensional analogues of Ricci-type connections)xn, as summarized in the theorem below. The proofs of the different statements listed here can be found in \cite{BourgeoisCahen:1998} (mainly section 5). See also \cite{BieliavskyEtAl:2005} section 3.1 for a coordinate-free formulation of some of the results in \cite{BourgeoisCahen:1998}.

\begin{theorem}\label{thm:PreferredSurfaceProperties}
    Let $(M,\omega)$ be a symplectic surface equipped with a preferred connection  $\nabla$. Then
    \begin{enumerate}[(i)]
        \item There exists a 1-form $u$ such that for all $X,Y,Z\in\mathfrak{X}(M)$ one has
        \begin{equation*}
            (\nabla_Xr)(Y,Z)=\omega(Y,X)u(Z)+\omega(Z,X)u(Y).
        \end{equation*}
        \item There exists a function $\beta$ such that
        \begin{equation*}
            \nabla u=\beta\omega.
        \end{equation*}
        \item Defining the vector field $\bar{u}$ by
        \begin{equation*}
            \iota(\bar{u})\omega=u.
        \end{equation*}
        There exist real numbers $A$ and $B$ such that
        \begin{equation*}
            \begin{split}
                r(\bar{u},\bar{u})&=\beta^2+B\\
                \frac 14 tr[\rho^2]&=\beta+A.
            \end{split}
        \end{equation*}
        \item Assume $d\beta\neq0$ and let $U=\{p\in M|r_p(\bar{u},\bar{u})\neq0\}$. Then the preferred connection is given on $U$ by
        \begin{equation}\label{eq:PreferredConnectionFormula}
            \begin{split}
                \nabla_{\bar{u}}\bar{u}&=-\beta\bar{u}\\
                \nabla_{\bar{u}}X_{\beta}&=-\beta X_{\beta}\\
                \nabla_{X_\beta}X_\beta&=(\beta^2+2A\beta-B)\bar{u}.
            \end{split}
        \end{equation}
        Conversely, given a 1-form $u$ and a non-constant function $\beta$ such that $u(X_\beta)=\beta^2+B$ for some real number $B$, equation \ref{eq:PreferredConnectionFormula} defines a preferred symplectic connection over $U:=\{p\in M|\beta^2\neq-B\}$.
        \item The symplectic connection $\nabla$ is locally symmetric if and only if $\beta=0$.
    \end{enumerate}
\end{theorem}

Let us now proceed with the construction of the examples. To avoid confusion between coordinate indices and exponents we denote the coordinates of $\mathbb{R}^2$ by $(x,y)$. We equip $\mathbb{R}^2$ with its standard constant symplectic form $\omega_0=dx\wedge dy$. We choose real numbers $A$ and $B$ with $B>0$ and endow $(\mathbb{R}^2,\omega_0)$ with the preferred symplectic connection induced by the function $\beta:=y$ and the 1-form $u:=(y^2+B)dx=(\beta^2+B)dx$, as described in (iv) of theorem \ref{thm:PreferredSurfaceProperties}.

We easily compute $X_{\beta}=\partial_x$ and $\bar{u}=-(y^2+B)\partial_y$. From this we verify that
\begin{equation*}
    u(X_\beta)=y^2+B=\beta^2+B
\end{equation*}
ensuring that $\beta$ and $u$ do indeed induce a preferred connection $\nabla$ on $\mathbb{R}^2$. Since $B>0$, $\beta^2+B$ never vanishes, so $\nabla$ is defined on the whole of $\mathbb{R}^2$. In accordance with theorem \ref{thm:PreferredSurfaceProperties} (iv) the Christoffel symbols of $\nabla$ are given by:
\begin{equation*}
    \begin{split}
        \nabla_{\partial_x}\partial_x&=-(y^2+B)(x^2+2Ay-B)\partial_y\\
        \nabla_{\partial_x}\partial_y&=\nabla_{\partial_y}\partial_x=\frac{y}{y^2+B}\partial_x\\
        \nabla_{\partial_y}\partial_y&=-\frac{y}{y^2+B}\partial_y.
    \end{split}
\end{equation*}
Clearly the Christoffel symbols are analytic, so $\nabla$ is analytic. Finally, since $\beta=y$ is non-zero (with nowhere vanishing differential) by theorem \ref{thm:PreferredSurfaceProperties} (v) $(M,\nabla)$ is not locally symmetric.


\end{document}